\documentclass[a4paper,10pt,leqno]{article}
\usepackage[T1]{fontenc}
\usepackage[utf8]{inputenc}
\usepackage{lmodern,microtype}
\usepackage{amsmath,amssymb,amsthm,mathtools,mathrsfs}
\usepackage[margin=1.13in]{geometry}
\usepackage{booktabs,tabularx,array,enumitem}
\usepackage{xcolor}
\usepackage{tikz}
\usepackage[colorlinks=true,linkcolor=black,citecolor=black,urlcolor=blue]{hyperref}
\hypersetup{pdftitle={Morse Index, Stability and Spectral Bounds for Periodic Orbits in Reversible Lagrangian Systems},pdfauthor={Alessandro Portaluri and Li Wu}}
\allowdisplaybreaks
\numberwithin{equation}{section}

\theoremstyle{plain}
\newtheorem{thm}{Theorem}[section]
\newtheorem{prop}[thm]{Proposition}
\newtheorem{lem}[thm]{Lemma}
\newtheorem{cor}[thm]{Corollary}
\theoremstyle{definition}

\theoremstyle{remark}
\newtheorem{rmk}[thm]{Remark}
\theoremstyle{plain}
\newtheorem{mainthm}{Theorem}

\newcommand{\R}{\mathbb R}
\newcommand{\C}{\mathbb C}

\newcommand{\dd}{\,\mathrm dt}
\newcommand{\Rev}{\mathcal R}

\newcommand{\cA}{\mathcal A}
\newcommand{\cJ}{\mathscr J}
\newcommand{\bform}{\mathfrak b}
\newcommand{\qK}{\mathcal Q_K}

\DeclareMathOperator{\Graph}{Gr}

\newcommand{\nupos}{\nu_{+}}
\newcommand{\nustar}{\nu_{+}^{*}}
\newcommand{\nubad}{\nu_{\mathrm{off}+}}
\newcommand{\nugrow}{\nu_{\mathrm u}^{+}}

\title{Morse Index, Stability and Spectral Bounds for Periodic Orbits in Reversible Lagrangian Systems}
\author{Alessandro Portaluri\and Li Wu}
\date{}
\begin{document}
\maketitle

\begin{abstract}
We prove quantitative stability and instability estimates for reversible
periodic solutions of strictly Legendre-convex Lagrangian systems.  If the
configuration dimension is $n$, the fixed-period Morse index is $k$, and
$k_D,\nu_D$ are the Dirichlet index and nullity on a half-period, set
\[
   \widetilde k=k-2k_D-\nu_D.
\]
Then the algebraic number $\nu_+$ of positive real Floquet multipliers satisfies
\[
   \nu_+\ge 2n-2\widetilde k,
\]
and we obtain corresponding bounds for the positive real spectrum away from
$1$, allowing arbitrary Jordan chains at the unit multiplier.  Hence at most
$2\widetilde k$ Floquet multipliers can lie outside the positive real axis.  In
particular, when $k=0$ we recover the positive-real-spectrum theorem for
reversible minimizers, while for nonconstant autonomous brake orbits the
bounds improve by two.  If, moreover, $n\ge2$ and the algebraic multiplicity of $1$ is two,
an index-one brake orbit is transversely hyperbolic, with $n-1$ stable and
$n-1$ unstable directions.  Decoupled oscillators show that the refined bounds
can be sharp even when the estimate based on $k$ alone is vacuous, and the
H\'enon--Heiles straight-line orbit provides a nonlinear illustration of the
half-period mechanism.
\end{abstract}



\section{Introduction and main results}\label{sec:introduction}

The main purpose of this paper is a stability question: how much of the
Floquet spectrum of a reversible periodic orbit is constrained by its
variational index?  We give quantitative lower bounds, with algebraic
multiplicity, for the positive real Floquet spectrum in terms of the
fixed-period Morse index and half-period Dirichlet data.  The estimates retain
the whole generalized eigenspace at the multiplier $1$ and therefore require
no semisimplicity assumption.

The relation between minimality and stability of periodic motions is classical,
but the two notions live in different categories.  The second variation of
the action is a self-adjoint Sturm--Liouville problem, whereas linear stability
is encoded by a symplectic monodromy matrix.  Reversibility links the two:
it decomposes periodic variations into reflection-even and reflection-odd
parts, identified with half-period Neumann and Dirichlet problems, and it
factorizes the full monodromy through the half-period evolution.  The
variational decompositions and boundary-index formulas needed for this step
are known; we recall them, with attribution, only to make the paper
self-contained and to fix the sign and nullity conventions.  The new part is
the quantitative algebraic estimate converting the resulting boundary inertia
into information on prescribed sectors of the Floquet spectrum.

A useful benchmark is the theorem of Ure\~na \cite{Ure18}: for a reversible
periodic minimizer all Floquet multipliers are positive real.  Positive real
spectrum does not itself mean spectral stability, since a reciprocal pair
$\lambda,\lambda^{-1}$ with $\lambda>1$ is hyperbolic.  Our first theorem
extends this spectral-location information beyond minimizers: when the
periodic Morse index is positive it quantifies how many multipliers are still
forced to remain positive real.  At $k=0$ it recovers Ure\~na's conclusion,
whereas for nonconstant autonomous brake orbits the time-translation Jacobi
field yields a stronger estimate and, in the index-one case under the minimal
algebraic multiplicity at $1$, transverse hyperbolicity.

The answer depends on more than the full-period index alone.  On the
half-period, some negative directions are already visible under fixed-endpoint
(Dirichlet) variations.  Removing precisely this contribution leads to the
quantity
\[
 \widetilde k=k-2k_D-\nu_D.
\]
The \emph{boundary-index part} of this mechanism is not new.  Hu, Wu and Yang
proved a Morse-index comparison for general self-adjoint boundary conditions
in \cite[Theorem~1.1]{HWY20}, formulated through the triple index; the
finite-dimensional triple-index identities used there are developed in
\cite{ZWZ18}.  In the present complementary Neumann--Dirichlet situation,
that general theory specializes to the identity
\[
 m^+(\qK)=k_N-k_D-\nu_D=k-2k_D-\nu_D.
\]
We nevertheless give a direct derivation because it fixes all signs and
nullity terms in our conventions and identifies explicitly the matrix
\(\qK=N-K^\top NK\) that enters the spectral argument.  No novelty is claimed
for this Morse-index comparison itself.  The contribution for which we claim
novelty is the subsequent quantitative \emph{algebraic} control of the
positive real Floquet spectrum by the inertia of this form, including the full
generalized eigenspace at the multiplier \(1\), and the resulting refined
stability/instability consequences.  The reflection splitting used in the
specialization is the \(\mathbb Z_2\)-part of the more general dihedral
decomposition of Hu, Portaluri and Yang \cite{HPY20}.

The relation between variational indices and linear instability for periodic
Lagrangian systems has also been developed in a series of works by Portaluri,
Wu and Yang.  In the non-autonomous setting, \cite{PWY21} gives a
spectral-flow/Maslov-index criterion for linear instability, while the
autonomous free-period problem is treated in \cite{PWY22}.  Those results use
parity and spectral-index information to detect instability.  The present
fixed-time reversible problem has a different emphasis: reversibility produces
an exact half-period boundary inertia, and that inertia yields quantitative
lower bounds on prescribed sectors of the Floquet spectrum.

For periodic brake orbits, the present viewpoint is closely related to the
more recent work of Asselle, Hu, Portaluri and Wu \cite{AHPW26}, where the
geometry at the brake instants is used to prove non-minimality and instability
results for natural Lagrangians.  Here we remain in the fixed-time
Legendre-convex setting and retain the half-period Dirichlet index and nullity
explicitly in the final spectral estimates.  This distinction is important:
the correction $2k_D+\nu_D$ contains information that can be invisible to the
full-period Morse index alone.

A short communication containing the core spectral estimate and a compressed
version of the argument has been submitted separately to \emph{Comptes Rendus
Math\'ematique} \cite{PWCR26}.  The present article is intended as the full
account of the theory. In comparison with that note, we make the relation of
the half-period identity with the pre-existing boundary-index theorem
\cite{HWY20} and triple-index formalism \cite{ZWZ18} completely explicit and
include a direct proof for the reader; the additional work is concentrated on
the stability side: we identify the concrete boundary matrix needed for the
Floquet count, treat the generalized eigenspace at the multiplier $1$ without
semisimplicity assumptions, clarify the distinction from the Lin--Zeng
instability index, and compute the indices and spectral counts in explicit
model families; see Section~\ref{sec:examples}.  In particular, a
three-dimensional resonant oscillator shows that the correction
$2k_D+\nu_D$ can turn a vacuous estimate based on $k$ alone into the exact
spectral count.  The H\'enon--Heiles $A$-orbit then provides a nonlinear
illustration in which the Jacobi equation splits into longitudinal and
transverse scalar problems and the autonomous Dirichlet null direction is
visible as the time-translation field.

A related but genuinely different instability-index theorem is due to Lin and
Zeng \cite{LZ22}.  Their framework concerns a time-independent Hamiltonian
generator $JL$, whereas linearization along a nonconstant periodic orbit in
our problem is a time-periodic Hamiltonian system and the relevant spectral
object is its monodromy matrix.  Their negative index is an energy index of the
fixed generator; ours comes from the fixed-period action Hessian and its
half-period realizations.  Section~\ref{sec:lin-zeng} explains why, without
additional structure, there is no canonical identification between these two
index theories.

We now introduce the objects appearing in the statements.  Particular care is
taken with boundary conditions, because periodic, Neumann and Dirichlet Morse
indices are indices of different self-adjoint realizations of the same formal
Jacobi operator.

\subsection{Setting and notation}

Let \(\mathcal O\subset\R^n\) be open and \(T>0\). Consider a
\(C^2\) Lagrangian \(L(t,q,v)\), defined on a neighbourhood of the phase
curve under consideration, with
\begin{equation}\label{eq:L-reversibility}
 L(t+T,q,v)=L(t,q,v),\qquad L(-t,q,-v)=L(t,q,v).
\end{equation}
Let \(x\) be a classical \(T\)-periodic solution of the Euler--Lagrange
equation satisfying \(x(T-t)=x(t)\). Set \(h=T/2\) and, along \(x\),
\begin{equation}\label{eq:PQS}
 P=L_{vv},\qquad Q=L_{vq},\qquad S=L_{qq}.
\end{equation}
We assume the strict Legendre condition along the orbit:
\begin{equation}\label{eq:Legendre}
 P(t)=P(t)^\top\geq\alpha I_n\quad\text{for some }\alpha>0.
\end{equation}
The matrices \(P,Q,S\) are continuous; \(S\) is symmetric. The second
variation of \(\cA(x)=\int_0^T L(t,x,\dot x)\dd\) is represented by
\begin{equation}\label{eq:index-form}
 \bform_I(y,w)=\int_I\bigl(
 \dot y^\top P\dot w+\dot y^\top Qw
 +y^\top Q^\top\dot w+y^\top Sw\bigr)\dd.
\end{equation}
Its quadratic form is \(\bform_I(y,y)\), without a factor of \(1/2\).
When growth assumptions on \(L\) are not imposed, the second variation is
first computed on smooth variations and then extended to \(H^1\) by
\eqref{eq:index-form}; no global \(C^2\) action functional on all of
\(H^1\) is assumed.

Write
\[
 H_P=\{y\in H^1([0,T],\R^n):y(T)=y(0)\},\quad
 H_N=H^1([0,h],\R^n),\quad H_D=H^1_0([0,h],\R^n).
\]
The periodic and half-period Morse indices are
\begin{equation}\label{eq:indices}
 k=m^-(\bform_{[0,T]}|_{H_P}),\qquad
 k_N=m^-(\bform_{[0,h]}|_{H_N}),\qquad
 k_D=m^-(\bform_{[0,h]}|_{H_D}).
\end{equation}
Here the negative index is the maximal dimension of a subspace on which the
quadratic form is negative definite. The Legendre condition and compactness
of the interval make these indices finite. Let \(\nu_N,\nu_D\) be the
nullities of the two half-period forms. The natural Neumann condition is
\(p_y=0\) at the endpoints, where
\begin{equation}\label{eq:quasiderivative}
 p_y=P\dot y+Qy,\qquad z_y=(p_y,y)^\top.
\end{equation}
Thus all indices in this paper refer to specified boundary conditions and to
fixed time.

The Jacobi equation is
\begin{equation}\label{eq:Jacobi}
 \ell y=-\dot p_y+Q^\top\dot y+Sy=0.
\end{equation}
Use phase coordinates \((p,q)\), the matrix and symplectic form
\[
 J=\begin{pmatrix}0&-I_n\\ I_n&0\end{pmatrix},\qquad
 \omega(u,v)=\langle Ju,v\rangle.
\]
Let \(\gamma(0)=I_{2n}\) be the fundamental matrix of the first-order
system equivalent to \eqref{eq:Jacobi}, and set
\[
 K=\gamma(h),\qquad M=\gamma(T),\qquad
 \nu_1=\dim\ker(M-I_{2n}),\qquad
 a_1=\dim\ker\bigl((M-I_{2n})^{2n}\bigr).
\]
In particular, \(\nu_1\) is geometric multiplicity and \(a_1\) is
algebraic multiplicity. They are not interchangeable. For algebraic spectral
counts write
\begin{align}
 \nupos&=\sum_{\lambda\in\sigma(M)\cap(0,\infty)}
       \operatorname{algmult}(\lambda),\label{eq:count-positive}\\
 \nustar&=\sum_{\lambda\in\sigma(M)\cap((0,\infty)\setminus\{1\})}
       \operatorname{algmult}(\lambda)=\nupos-a_1,\label{eq:count-positive-star}\\
 \nubad&=2n-\nupos.\label{eq:count-bad}
\end{align}
Negative lower bounds for these nonnegative counts are understood as vacuous.

\subsection{Main estimates}

The first theorem concerns a general time-periodic reversible system.  It
should be compared with Ure\~na's positive-real-spectrum theorem for reversible
minimizers \cite{Ure18} and with the brake-orbit stability estimates of Hu,
Wu and Yang \cite[Theorems~1.8--1.9]{HWY20}.  The statement below is formulated
instead as an algebraic lower bound for the positive real Floquet spectrum at
arbitrary Morse index, with an explicit correction by \(k_D\) and \(\nu_D\)
and with Jordan chains at \(1\) retained.  It is useful to read the
inequalities from left to right: the actual spectral count is bounded first by
the sharpened half-period quantity \(\widetilde k\), and then by the coarser
full-period Morse index \(k\).

\begin{mainthm}[Reversible periodic systems]\label{thm:non-autonomous}
Under \eqref{eq:L-reversibility}--\eqref{eq:Legendre}, set
\[
 \widetilde k:=k-2k_D-\nu_D.
\]
Then \(\widetilde k\geq0\), and the algebraic spectral counts satisfy the
chains
\begin{align}
 \nupos
 &\geq 2n-2\widetilde k
  =2n-2k+4k_D+2\nu_D
 \geq 2n-2k,\label{eq:main-positive-chain}\\
 \nustar
 &\geq 2n-2\widetilde k-2\nu_1
  =2n-2k+4k_D+2\nu_D-2\nu_1
 \geq 2n-2k-2\nu_1.\label{eq:main-star-chain}
\end{align}
If the algebraic multiplicity \(a_1\) of the multiplier \(1\) is known, one
also has
\begin{equation}\label{eq:main-algebraic}
 \nustar\geq2n-2\widetilde k-a_1.
\end{equation}
\end{mainthm}

At \(k=0\), the first chain forces all \(2n\) multipliers to be positive
real. Thus the theorem recovers, in the present Legendre-convex formulation,
the spectral conclusion of Ure\~na's theorem for reversible minimizers
\cite[Theorem~1.1]{Ure18}.  More generally, the number \(\nubad\) of multipliers outside the
positive real axis satisfies \(\nubad\leq2\widetilde k\leq2k\).  Thus
\(\widetilde k\) is the part of the Morse index that is actually visible to
this spectral estimate.

For an autonomous brake orbit there is an additional geometric Jacobi field:
time translation produces \(\dot x\), and the brake symmetry makes this field
Dirichlet at the two brake instants.  This supplies one unit of Dirichlet
nullity and improves the preceding chains.

\begin{mainthm}[Nonconstant autonomous brake orbits]\label{thm:autonomous}
Assume in addition that \(L\) is autonomous and \(x\) is nonconstant.  Then
\(\nu_D\geq1\), \(k\geq1\), and
\begin{align}
 \nupos&\geq2n-2\widetilde k\geq2n-2k+2,
 \label{eq:autonomous-positive-chain}\\
 \nustar&\geq2n-2\widetilde k-2\nu_1
 \geq2n-2k+2-2\nu_1.
 \label{eq:autonomous-star-chain}
\end{align}
If \(a_1=2\), then
\begin{equation}\label{eq:autonomous-a1-two}
 \nustar\geq2n-2k.
\end{equation}
\end{mainthm}

\begin{rmk}[Why nonconstancy is necessary]\label{rmk:equilibrium}
For \(L(q,v)=\tfrac12v^2-\tfrac12\omega^2q^2\), the equilibrium
\(x\equiv0\), viewed as a periodic solution with \(T=\pi/\omega\), has
\(k=1\) and \(M=-I_2\). It has no positive real multiplier, so the first
bound in \eqref{eq:autonomous-positive-chain}--\eqref{eq:autonomous-star-chain} would be false if equilibria were
included. Theorem~\ref{thm:non-autonomous} does apply to this equilibrium.
\end{rmk}

For the time-reversal symmetry used here, \(x(T-t)=x(t)\) implies
\(\dot x(0)=\dot x(h)=0\). In the autonomous case these are brake instants.
The Jacobi field \(\dot x\) then belongs to the half-period Dirichlet
kernel. Its contribution to the Dirichlet nullity accounts directly for the
improvement in Theorem~\ref{thm:autonomous}.

The proof can now be summarized geometrically.  Reflection splits the
periodic variation space into an even part and an odd part, producing
\(k=k_N+k_D\).  The remaining difference between these two half-period
problems is finite-dimensional: it is recorded by the endpoint flux of Jacobi
fields.  If \(K\) transports Cauchy data from one brake half-period endpoint
to the other, the symmetric form
\(\qK=N-K^\top NK\) measures the defect of this transport with respect to the
Neumann--Dirichlet boundary pairing \(N\).  Its positive inertia is exactly
\begin{equation}\label{eq:central-identity-intro}
 m^+(\qK)=k_N-k_D-\nu_D=k-2k_D-\nu_D.
\end{equation}
As emphasized above, the Morse-index identity itself should be viewed as the
Neumann--Dirichlet specialization of \cite[Theorem~1.1]{HWY20}, together with
the triple-index formula of \cite[Lemma~3.13]{ZWZ18}; what will matter here is
its explicit realization by the matrix \(\qK\).  Equivalently, \(\qK\) is the quadratic form induced on \(\Graph K\) by the
pair of transverse Lagrangian boundary spaces corresponding to Neumann and
Dirichlet conditions.  Reversibility then identifies
\(\qK=N(I-M)\).  On the spectral subspace outside the positive real axis
this form is nondegenerate and has balanced inertia; a separate nilpotent
argument treats the generalized eigenspace at \(1\).  No semisimplicity
assumption is used.

\section{The variational identity on a half-period}\label{sec:variational}

This section extracts from the infinite-dimensional second variation the
finite-dimensional quantity that will later control the Floquet spectrum.
Most of the variational ingredients in this section are standard, or are
special cases of the general boundary-condition Morse-index theory in
\cite{HWY20} and the symmetry decomposition in \cite{HPY20}.  We retain the
direct proofs to keep the paper self-contained and, more importantly, to make
our sign, nullity and boundary-pairing conventions completely explicit.
There are three steps.  First we specify the self-adjoint realizations of the
Jacobi operator and use reflection to split the periodic problem.  Next we
compare the Neumann and Dirichlet indices through the boundary values of
Jacobi fields.  Finally we identify that boundary form with \(\qK\) and with
the corresponding triple-index form in the Lagrangian Grassmannian.

\subsection{Analytic realization and reflection splitting}

We begin at the operator level in order to remove any ambiguity about what is
meant by the periodic, Neumann and Dirichlet indices.  The differential
expression is the same in all three cases; only the form domain and the
operator boundary conditions change.

The coefficient symmetries inherited from \(L\) and \(x\) are
\begin{equation}\label{eq:coefficient-reflection}
 P(T-t)=P(t),\qquad S(T-t)=S(t),\qquad Q(T-t)=-Q(t).
\end{equation}
The form \(\bform_I\) satisfies a G\aa rding inequality
\(\bform_I(y,y)\geq c\|\dot y\|_{L^2}^2-C\|y\|_{L^2}^2\),
with \(c>0\). Adding a sufficiently large multiple of the \(L^2\) norm
makes it coercive on each form domain. The compact embedding
\(H^1(I)\hookrightarrow L^2(I)\) therefore yields a self-adjoint
realization with compact resolvent and the indices in \eqref{eq:indices}.

With only continuous coefficients the natural operator domain is expressed
through the quasiderivative \(p_y\).  We denote by
\(\mathscr L_{\mathrm{per}}\), \(\mathscr L_N\), and \(\mathscr L_D\) the
three self-adjoint realizations of the formal Jacobi expression \(\ell\):
\begin{align}
 \mathcal D(\mathscr L_{\mathrm{per}})
 &=\{y\in H^1([0,T],\R^n):p_y\in H^1,
      \ y(T)=y(0),\ p_y(T)=p_y(0)\},\label{eq:operator-periodic}\\
 \mathcal D(\mathscr L_N)
 &=\{y\in H^1([0,h],\R^n):p_y\in H^1,
      \ p_y(0)=p_y(h)=0\},\label{eq:operator-neumann}\\
 \mathcal D(\mathscr L_D)
 &=\{y\in H^1([0,h],\R^n):p_y\in H^1,
      \ y(0)=y(h)=0\},\label{eq:operator-dirichlet}
\end{align}
with \(\mathscr L_{\bullet}y=\ell y\) on the corresponding domain.  The
Morse indices \(k,k_N,k_D\) in \eqref{eq:indices} are precisely the numbers
of negative eigenvalues, counted with multiplicity, of
\(\mathscr L_{\mathrm{per}},\mathscr L_N,\mathscr L_D\), respectively;
the nullities are the dimensions of their kernels.  Thus ``Neumann'' and
``Dirichlet'' below always refer to these operator realizations, not merely to
informal endpoint labels.

An \(H^2\) description would require additional coefficient regularity and
is not needed here.  A solution of \eqref{eq:Jacobi} is uniquely determined
by its Cauchy datum \((p_y(t_0),y(t_0))\) at any time \(t_0\), because the
associated first-order system has continuous coefficients.

The following reflection splitting is not a new result.  It is the
\(\mathbb Z_2\)-reflection component of the dihedral-equivariant decomposition
of the path space developed in \cite[Section~3]{HPY20}; see in particular the
additive Morse-index decomposition obtained by restricting the Hessian to the
reflection-invariant and reflection-anti-invariant subspaces.  In the present
brake-orbit notation these two subspaces identify, after restriction to
\([0,T/2]\), with the Neumann and Dirichlet half-period problems,
respectively.  The same Neumann/Dirichlet splitting is used in the brake-orbit
analysis of \cite{HWY20}.  Thus the Morse-index equality in
\eqref{eq:splitting} is a specialization of the earlier equivariant
index decomposition, and the nullity equality follows from the same
orthogonal splitting.  We reproduce the short proof only for
self-containment and to identify precisely the quasiderivative boundary
conditions used here. 

\begin{lem}[Reflection splitting]\label{lem:splitting}
Let \(\mathscr T:H_P\to H_P\) be the reflection
\((\mathscr T y)(t)=y(T-t)\).  Define
\[
 H_P^+=\ker(\mathscr T-I),\qquad H_P^-=\ker(\mathscr T+I).
\]
For every \(y\in H_P\) write
\[
 y_+:=\frac12(y+\mathscr T y),\qquad
 y_-:=\frac12(y-\mathscr T y),
\]
so that \(y=y_++y_-\), with \(y_+\in H_P^+\) and
\(y_-\in H_P^-\).  Then
\begin{equation}\label{eq:splitting}
 k=k_N+k_D,\qquad \nu_1=\nu_N+\nu_D.
\end{equation}
\end{lem}
\begin{proof}
By \eqref{eq:coefficient-reflection}, \(\mathscr T\) preserves the periodic
form domain and
\(
 \bform_{[0,T]}(\mathscr T y,\mathscr T w)=\bform_{[0,T]}(y,w).
\)
Hence \(H_P^+\) and \(H_P^-\) are form-orthogonal: if
\(y_+\in H_P^+\) and \(y_-\in H_P^-\), then
\[
 \bform(y_+,y_-)
 =\bform(\mathscr T y_+,\mathscr T y_-)
 =-\bform(y_+,y_-),
\]
so the mixed term vanishes.

We next identify the two summands with the half-period form spaces.  The
restriction map
\(
 r_+:H_P^+\to H_N,\ y\mapsto y|_{[0,h]},
\)
is an isomorphism.  Its inverse is the even reflection of an arbitrary
\(u\in H^1([0,h],\R^n)\): set \(y(t)=u(t)\) for \(0\le t\le h\) and
\(y(t)=u(T-t)\) for \(h\le t\le T\).  At the level of the operator kernel, evenness gives
$\dot y(0)=\dot y(h)=0$.  Moreover
\eqref{eq:coefficient-reflection}, together with periodicity, implies
$Q(0)=Q(h)=0$.  Hence the quasiderivative
$p_y=P\dot y+Qy$ satisfies
\(p_y(0)=p_y(h)=0\), which are precisely the natural Neumann boundary
conditions.  Thus the nullspace on this summand is exactly
\(\ker\mathscr L_N\).

Similarly, the restriction
\(
 r_-:H_P^-\to H_D
\)
is an isomorphism.  Indeed, oddness with respect to \(t=0\) (using
periodicity) gives \(y(0)=0\), while oddness with respect to the midpoint
gives \(y(h)=0\).  Conversely every \(u\in H_0^1([0,h],\R^n)\) has a unique
odd reflected periodic extension.  On both summands the full-period form is
twice the corresponding half-period form:
\[
 \bform_{[0,T]}(y_\pm,y_\pm)
 =2\,\bform_{[0,h]}(r_\pm y_\pm,r_\pm y_\pm).
\]
Therefore negative indices and nullities add, proving the first equality and
\(
 \dim\ker\mathscr L_{\mathrm{per}}=\nu_N+\nu_D.
\)
Finally, a periodic Jacobi field lies in
\(\ker\mathscr L_{\mathrm{per}}\) exactly when its initial phase datum is
fixed by the monodromy, so
\(\dim\ker\mathscr L_{\mathrm{per}}=\dim\ker(M-I)=\nu_1\).
\end{proof}

The direct formulation above is useful here because it applies at degenerate
orbits and does not presume that \(Q(t)\) vanishes away from the brake
instants.

\subsection{A boundary inertia formula}

The reflection splitting reduces the periodic problem to two half-period
problems, but it does not yet explain their difference.  We now isolate that
difference.  The key observation is that after fixing the endpoints, the
remaining finite-dimensional degrees of freedom are precisely the boundary
values of Jacobi fields.  Their boundary flux carries the missing inertia.

For this subsection put \(H=H_N\), \(D=H_D\), and
\(\bform=\bform_{[0,h]}\). The space
\[
 \cJ=\{y\in H:\bform(y,w)=0\text{ for every }w\in D\}
\]
is precisely the \(2n\)-dimensional space of Jacobi fields. Its elements
satisfy \eqref{eq:Jacobi}, initially in the distributional sense and then in
the first-order sense through \(p_y\). Green's identity gives
\begin{equation}\label{eq:boundary-form}
 \bform(y,w)=\bigl[p_y(t)^\top w(t)\bigr]_0^h,
 \qquad y\in\cJ,\quad w\in H.
\end{equation}

At the abstract level, the next reduction is also not new.  The general
boundary-condition Morse-index theorem of Hu--Wu--Yang
\cite[Theorem~1.1]{HWY20} gives
\[
 m^-(\Lambda_0)-m^-(\Lambda_D)
 =i(\operatorname{Gr}\gamma(h),\Lambda_0,\Lambda_D)
\]
for an arbitrary self-adjoint boundary condition \(\Lambda_0\).  Taking
\(\Lambda_0=\Lambda_N\) and writing the resulting triple-index form in terms
of endpoint values of Jacobi fields gives precisely the fixed-endpoint inertia
reduction \eqref{eq:boundary-inertia} below; equivalently, it follows from
\cite[Theorem~1.1]{HWY20} together with the finite-dimensional triple-index
formula of \cite[Lemma~3.13]{ZWZ18}.  Formula \eqref{eq:boundary-inertia} is
therefore recalled here in a concrete Hilbert-space form, not claimed as a new
index theorem.  We give a direct proof because the term \(\nu_D\) is easy to
lose if one argues only in the nondegenerate case, and this term is essential
for the autonomous improvement later on.

\begin{lem}[Inertia with fixed endpoints]\label{lem:boundary-inertia}
The restricted form on Jacobi fields satisfies
\begin{equation}\label{eq:boundary-inertia}
 k_N=k_D+\nu_D+m^-(\bform|_{\cJ}).
\end{equation}
\end{lem}
\begin{proof}
Let \(Z=\ker(\bform|_D)\), so \(\dim Z=\nu_D\). Choose a closed
complement \(D_0\) of \(Z\) in \(D\) on which \(\bform\) is
nondegenerate. Such a complement exists because the form operator on \(D\)
is self-adjoint Fredholm. Its negative index is \(k_D\). Consequently
\[
 H=D_0\oplus_{\bform}W,\qquad
 W=D_0^{\perp_{\bform}},\qquad \dim W=2n+\nu_D.
\]
The space \(Z\subset W\) is isotropic. Moreover
\(Z\cap\ker(\bform|_H)=\{0\}\): a field in this intersection satisfies
both Dirichlet and Neumann conditions, hence has zero initial phase value
and vanishes by uniqueness.

It follows that the pairing of \(Z\) with \(W\) has rank \(\nu_D\).
Choose a subspace \(F\subset W\) of dimension \(\nu_D\) on which this
pairing is nonsingular. In appropriate bases the form on \(Z\oplus F\)
has matrix
\[
 \begin{pmatrix}0&I_{\nu_D}\\ I_{\nu_D}&C\end{pmatrix},\qquad C=C^\top,
\]
and therefore inertia \((\nu_D,\nu_D,0)\), as follows by replacing the
\(F\)-basis by its difference with one half of the corresponding
\(ZC\)-basis. Let \(G\) be its form-orthogonal complement in \(W\).
Since \(D=D_0\oplus Z\), orthogonality to \(D\) now gives
\(\cJ=Z\oplus G\), and \(Z\) is in the radical of this restricted
form. Thus
\[
 m^-(\bform|_H)=k_D+\nu_D+m^-(\bform|_G),\qquad
 m^-(\bform|_{\cJ})=m^-(\bform|_G),
\]
which proves the formula.
\end{proof}

Thus \eqref{eq:boundary-inertia} should be read as a direct, self-contained
rederivation of the Neumann--Dirichlet specialization of the earlier general
boundary-condition index theorem \cite[Theorem~1.1]{HWY20}, with the
Dirichlet nullity contribution displayed explicitly.

The term \(\nu_D\) has a concrete origin: allowing the endpoints to vary
pairs each Dirichlet null direction with a boundary direction, contributing
one positive and one negative direction. Omitting this contribution would
lose both the exact identity and the autonomous improvement.

\subsection{Identification with the reversible form and the triple index}

We now express the boundary inertia in phase-space language.  This is the
point where the variational problem becomes finite-dimensional: a Jacobi
field is represented by its initial Cauchy datum, and the half-period map
\(K\) transports this datum to the other endpoint.

Set
\begin{equation}\label{eq:RN}
 \Rev=\begin{pmatrix}-I_n&0\\0&I_n\end{pmatrix},\qquad
 N=\begin{pmatrix}0&I_n\\I_n&0\end{pmatrix},\qquad
 \qK:=N-K^\top NK.
\end{equation}
The letters \(S(t)\) and \(\Rev\) distinguish the position Hessian from
the reversing involution.  The matrix \(N\) represents the symmetric
boundary pairing \((p,q)\mapsto2p^\top q\).  Consequently \(\qK\) has a
direct geometric meaning: it measures how much this pairing changes after
transporting Cauchy data through one half-period by \(K\).  Thus \(\qK=0\)
would mean that \(K\) preserves this Neumann--Dirichlet polarization exactly,
whereas the positive and negative inertia of \(\qK\) measure the two possible
signs of the boundary-flux defect.  Formula \eqref{eq:QK-boundary-congruence}
below makes this interpretation precise. For \(z=(p_0,q_0)\), let \(y_z\) be the Jacobi
field with initial phase value \(z\). Then \(z_{y_z}(h)=Kz\), and
\eqref{eq:boundary-form} gives
\begin{equation}\label{eq:QK-boundary-congruence}
 \bform(y_z,y_z)=p_h^\top q_h-p_0^\top q_0
       =-\tfrac12 z^\top \qK z.
\end{equation}
The map \(z\mapsto y_z\) is an isomorphism onto \(\cJ\).

The following proposition is therefore \emph{recalled in specialized form},
not proposed as a new Morse-index theorem.  As an index identity it is the
concrete Neumann--Dirichlet specialization of \cite[Theorem~1.1]{HWY20}, or,
equivalently, of that theorem combined with the finite-dimensional triple-index
formula \cite[Lemma~3.13]{ZWZ18}.  We include a complete direct proof solely
for self-containment and because the later spectral argument needs the
particular matrix representative \(\qK\), its sign convention, and the
nullity term \(\nu_D\) explicitly.

\begin{prop}[Neumann--Dirichlet boundary identity]\label{prop:exact-identity}
One has
\begin{equation}\label{eq:exact-identity}
 m^+(\qK)=k_N-k_D-\nu_D=k-2k_D-\nu_D.
\end{equation}
In particular, \(0\leq m^+(\qK)\leq k-\nu_D\leq k\).
\end{prop}
\begin{proof}
By \eqref{eq:QK-boundary-congruence},
\(m^-(\bform|_{\cJ})=m^+(\qK)\). Apply
Lemma~\ref{lem:boundary-inertia} and then Lemma~\ref{lem:splitting}.
\end{proof}

To compare with the index-theoretic formulation, use the doubled symplectic
space with form \(\Omega=-\omega\oplus\omega\), and write
\[
 L_N=\{0\}\oplus\R^n,\quad L_D=\R^n\oplus\{0\},\qquad
 \Lambda_N=L_N\oplus L_N,\quad \Lambda_D=L_D\oplus L_D.
\]
For Lagrangians \(\alpha,\beta,\delta\), use the triple form
\(\mathcal Q(\alpha,\beta;\delta)\) on
\(\alpha\cap(\beta+\delta)\), with convention
\[
 \mathcal Q(u,v)=\Omega(u_\beta,v_\delta),\qquad
 u=u_\beta+u_\delta,\quad v=v_\beta+v_\delta.
\]
This is a well-defined symmetric form. The corresponding triple index is
\cite[Lemma 3.13]{ZWZ18}
\begin{equation}\label{eq:triple-index}
 i(\alpha,\beta,\delta)
 =m^+(\mathcal Q(\alpha,\beta;\delta))
   +\dim(\alpha\cap\delta)-\dim(\alpha\cap\beta\cap\delta).
\end{equation}
In the present case \(\Lambda_N\) and \(\Lambda_D\) are complementary.
Under \(z\mapsto(z,Kz)\), direct computation gives
\begin{equation}\label{eq:triple-congruence}
 \mathcal Q(\Graph K,\Lambda_N;\Lambda_D)((z,Kz),(z,Kz))
 =q_0^\top p_0-q_h^\top p_h=\tfrac12z^\top \qK z.
\end{equation}
Also \(\dim(\Graph K\cap\Lambda_D)=\nu_D\). Hence
\eqref{eq:triple-index} gives \(i(\Graph K,\Lambda_N,\Lambda_D)
=m^+(\qK)+\nu_D=k_N-k_D\).  This is precisely the specialization of
\cite[Theorem~1.1]{HWY20} to the complementary Neumann and Dirichlet boundary
spaces.  In particular, our direct calculation is included not to reprove
that general theorem, but to identify its finite-dimensional form with
\(\qK\) and to fix the sign and the factor of two in
\eqref{eq:triple-congruence}.

\section{Spectral bounds for a reversible symplectic matrix}\label{sec:algebra}

The preceding section produced the number \(m^+(\qK)\) from the Morse index.
We now ask what this inertia says about the spectrum of the monodromy.  From
this point until Section~\ref{sec:proofs}, the argument is purely
finite-dimensional.  Several ingredients---reciprocal spectral symmetry of
symplectic matrices and orthogonality of root subspaces for a self-adjoint
operator in an indefinite inner-product space---are standard; see, for
example, \cite[Chapter~1]{Lon02} and \cite[Chapters~2 and~6]{Bog74}.  We give
short proofs in our notation because the particular form \(N(I-M)\) and the
Jordan contribution at \(1\) are central to the quantitative estimates.
Reversibility gives a factorization of \(M\); the form \(\qK=N(I-M)\) then
behaves like an indefinite metric adapted to the spectral decomposition.  The
only delicate sector is the generalized eigenspace at \(1\).

\subsection{Factorization and spectral orthogonality}

We first record the algebraic identities forced by reversibility and show
that generalized spectral subspaces are orthogonal for the symmetric form
\(N\).  This is the mechanism that allows the inertia of \(\qK\) to be
computed sector by sector.

The first-order system associated with \eqref{eq:Jacobi} is
\begin{equation}\label{eq:Hamiltonian-system}
 \dot z=JB(t)z,\qquad
 B(t)=\begin{pmatrix}
 P^{-1}&-P^{-1}Q\\
 -Q^\top P^{-1}&Q^\top P^{-1}Q-S
 \end{pmatrix}=B(t)^\top.
\end{equation}
Thus \(\gamma(t)\) is symplectic.  The reversible half-period factorization
below is standard in the analysis of symmetric periodic and brake orbits;
compare \cite{HPY20,HWY20}.  Reflection of the solutions under
\eqref{eq:coefficient-reflection} shows that evolution over the second half
of the period is \(\Rev K^{-1}\Rev\), and therefore
\begin{equation}\label{eq:factorization}
 M=\Rev K^{-1}\Rev K.
\end{equation}
The remainder of this section is purely algebraic: \(K\) can be any real
symplectic matrix, with \(M\) defined by \eqref{eq:factorization} and
\(\qK\) by \eqref{eq:RN}.

\begin{lem}[Basic identities]\label{lem:basic-identities}
One has
\begin{gather}
 M=NK^\top NK,\qquad NM=M^\top N,\qquad
 \Rev M\Rev=M^{-1},\label{eq:basic-identities}\\
 \qK=N(I-M),\qquad \ker \qK=\ker(I-M).\label{eq:kernel-identity}
\end{gather}
\end{lem}
\begin{proof}
Since \(K^{-1}=-JK^\top J\), \(\Rev J=N\), and \(J\Rev=-N\),
we have \(\Rev K^{-1}\Rev=NK^\top N\). Consequently
\(NM=K^\top NK\) is symmetric. The reversing identity follows directly
from the product \eqref{eq:factorization}. Finally
\(\qK=N-NM=N(I-M)\), and \(N\) is invertible.
\end{proof}

For \(\lambda\in\sigma(M)\), let
\(E_\lambda^{\C}=\ker\bigl((M-\lambda I)^{2n}\bigr)\) in \(\C^{2n}\).
For a conjugation-invariant spectral set \(\Sigma\), let \(E(\Sigma)\)
be the real spectral subspace whose complexification is
\(\bigoplus_{\lambda\in\Sigma}E_\lambda^{\C}\). Thus
\[
 \dim_{\R}E(\Sigma)=\sum_{\lambda\in\Sigma}\dim_{\C}E_\lambda^{\C}.
\]
All inertia computations below are on real spaces. Define
\begin{equation}\label{eq:real-spectral-spaces}
 E_+=E(\sigma(M)\cap(0,\infty)),\quad
 E_1=\ker\bigl((M-I)^{2n}\bigr),\quad E_+^*=E(\sigma(M)\cap((0,\infty)\setminus\{1\})),
\end{equation}
and \(E_{\mathrm{off}+}=E(\sigma(M)\setminus(0,\infty))\). Then
\begin{equation}\label{eq:spectral-sum}
 \R^{2n}=E_{\mathrm{off}+}\oplus E_+^*\oplus E_1,\qquad
 \dim E_+=\nupos,\quad\dim E_+^*=\nustar,\quad\dim E_1=a_1.
\end{equation}

The root-space orthogonality used next is a standard fact for operators that
are self-adjoint with respect to a nondegenerate indefinite inner product; see
for instance \cite[Chapter~2]{Bog74}.  We record the finite-dimensional proof
to cover generalized eigenspaces explicitly.

\begin{lem}[Orthogonality and nondegenerate restrictions]\label{lem:orthogonality}
For the Hermitian extension \([u,v]_N=\bar u^\top Nv\),
\[
 [E_\lambda^{\C},E_\mu^{\C}]_N=0\quad\text{if }\bar\lambda\ne\mu.
\]
Consequently the decomposition \eqref{eq:spectral-sum} is orthogonal for
both \(N\) and \(\qK\), and the restriction of \(N\) to each summand
is nondegenerate.
\end{lem}
\begin{proof}
The identity \(NM=M^\top N\) makes \(M\) self-adjoint for
\([\cdot,\cdot]_N\). If \((M-\lambda I)^ru=0\) and
\(\bar\lambda\ne\mu\), the operator \((M-\bar\lambda I)^r\)
is invertible on \(E_\mu^{\C}\). Write
\(v=(M-\bar\lambda I)^rw\) there. Then
\[
 [u,v]_N=[(M-\lambda I)^ru,w]_N=0.
\]
The spectral sets in \eqref{eq:spectral-sum} are disjoint and closed under
conjugation, giving real \(N\)-orthogonality. If the restriction to one
summand had a radical, that radical would be orthogonal to the whole space,
contradicting invertibility of \(N\). Finally \(I-M\) preserves every
summand, so \(\qK=N(I-M)\) gives the corresponding orthogonality for
\(\qK\).
\end{proof}

The last nondegeneracy assertion is needed: in general, a subspace may avoid
the kernel of a symmetric matrix while the restricted form is degenerate.
Here nondegeneracy follows from the orthogonal spectral decomposition.

\subsection{Balanced inertia outside the positive real axis}

Away from the positive real axis, the form \(\qK\) cannot become singular
along the homotopy \(N(I-tM)\).  Reversibility then forces its positive and
negative indices to balance.  The underlying reciprocal and indefinite-metric
spectral symmetries are standard (see \cite{Bog74,Lon02}); the concise homotopy
formulation below is included because it packages all spectral types in the
exact form needed for our counting argument.

\begin{lem}[Balanced complementary sector]\label{lem:balanced}
The restriction of \(\qK\) to \(E_{\mathrm{off}+}\) is nondegenerate and
\begin{equation}\label{eq:balanced}
 m^+(\qK|_{E_{\mathrm{off}+}})
 =m^-(\qK|_{E_{\mathrm{off}+}})=\tfrac12\dim E_{\mathrm{off}+}.
\end{equation}
\end{lem}
\begin{proof}
Consider the real symmetric forms
\[
 Q_t=N(I-tM)|_{E_{\mathrm{off}+}},\qquad 0\leq t\leq1.
\]
Since this spectral subspace contains no positive real eigenvalue,
\(I-tM\) is invertible on it for every \(t\in[0,1]\).
Lemma~\ref{lem:orthogonality} then implies that each restricted form
\(Q_t\) is nondegenerate, so its inertia is constant.

The reversing identity maps each generalized eigenspace for \(\lambda\)
to that for \(\lambda^{-1}\). The set defining \(E_{\mathrm{off}+}\)
is invariant under reciprocation, hence this real subspace is
\(\Rev\)-invariant. Since \(\Rev^\top N\Rev=-N\), the restricted forms
\(N\) and \(-N\) are congruent. Nondegeneracy therefore forces equal
positive and negative indices. The same inertia holds at \(t=1\).
\end{proof}

This single homotopy treats negative real eigenvalues, \(-1\), non-real
unit-circle pairs, and complex quadruples, with all their Jordan blocks.
Equivalently, the path \(Q_t\) on this sector has no zero crossing and zero
spectral flow.  This is an elementary indefinite-metric reformulation of the
spectral separation underlying earlier reversible stability arguments such as
\cite{Ure18,HWY20}; separate normal-form arguments for the individual spectral
types are unnecessary here.

\subsection{The generalized eigenspace at one}

The multiplier \(1\) requires separate treatment because it is exactly the
kernel of \(\qK\).  Jordan chains at \(1\) may be arbitrarily long, so a
semisimple argument would lose information.  We instead build a large
isotropic subspace from the nilpotent part of \(M-I\), which relates algebraic
multiplicity to the geometric nullity.

Only \(E_1\) can carry a radical of \(\qK\). On this space set
\(A=(M-I)|_{E_1}\). Then \(A\) is nilpotent,
\(NA=A^\top N\), and \(\qK|_{E_1}=-NA\).

\begin{lem}[An isotropic subspace at one]\label{lem:folding}
The real subspace
\begin{equation}\label{eq:folding-space}
 V_1=\sum_{j\geq0}A^j(\ker A^{2j+1})\subset E_1
\end{equation}
is \(\qK\)-isotropic. If the Jordan block lengths of \(A\) are
\(\ell_1,\ldots,\ell_{\nu_1}\), then
\begin{equation}\label{eq:folding-dimension}
 \dim V_1=\sum_{r=1}^{\nu_1}\left\lceil\frac{\ell_r}{2}\right\rceil
 \geq\frac{a_1}{2}.
\end{equation}
\end{lem}
\begin{proof}
Take \(u=A^ix\), \(x\in\ker A^{2i+1}\), and
\(v=A^jy\), \(y\in\ker A^{2j+1}\). Self-adjointness of \(A\) for
\([\cdot,\cdot]_N\) gives
\[
 u^\top \qK v=-[A^ix,A^{j+1}y]_N
             =-[x,A^{i+j+1}y]_N=-[A^{i+j+1}x,y]_N.
\]
If \(i\geq j\), the first expression on the right vanishes; if
\(j\geq i\), the second does. Bilinearity proves isotropy of the sum.

On a Jordan chain \(e_1,\ldots,e_\ell\), with \(Ae_1=0\) and
\(Ae_r=e_{r-1}\), the contribution to \eqref{eq:folding-space} is
\[
 \operatorname{span}(e_1,\ldots,e_{\lceil\ell/2\rceil}).
\]
Taking the direct sum over all chains gives the dimension formula.
There are \(\nu_1\) chains, since \(\nu_1=\dim\ker A\).
\end{proof}

For a real symmetric form of inertia \((p,q,r)\), every isotropic subspace
has dimension at most \(\min(p,q)+r\): quotient by the radical and project
onto either definite factor. By Lemma~\ref{lem:orthogonality} and
\eqref{eq:kernel-identity}, the radical of \(\qK|_{E_1}\) is exactly
\(\ker(M-I)\), of dimension \(\nu_1\). Hence
Lemma~\ref{lem:folding} implies
\begin{equation}\label{eq:one-sector-bound}
 a_1\leq2m^+(\qK|_{E_1})+2\nu_1.
\end{equation}
This estimate is the step that allows a geometric nullity to control an
algebraic multiplicity without assuming semisimplicity.

The next statement is the main finite-dimensional stability estimate used in
this paper.  It should be distinguished from the previously known brake-orbit
bounds of Hu, Wu and Yang, which control geometric dimensions of spectral
subspaces \cite[Theorems~1.8--1.9]{HWY20}, and from Ure\~na's minimizer result
\cite[Theorem~1.1]{Ure18}.  Here the count is made with \emph{algebraic}
multiplicity, the positive real sector is singled out, and arbitrary Jordan
chains at \(1\) are retained.  This proposition, together with the known
boundary-index identity recalled above, is what yields
Theorems~\ref{thm:non-autonomous} and \ref{thm:autonomous}.

\begin{prop}[Universal spectral estimates]\label{prop:universal}
For every factorization \eqref{eq:factorization}, put \(p=m^+(\qK)\).
Then
\begin{equation}\label{eq:universal}
 \nupos\geq2n-2p,\qquad
 \nustar\geq2n-2p-2\nu_1,\qquad
 \nustar\geq2n-2p-a_1.
\end{equation}
In particular, the algebraic number of negative real multipliers is at most
\(2p\). If \(p=0\), the whole spectrum is positive real, including when
\(1\) is an eigenvalue.
\end{prop}
\begin{proof}
By Lemma~\ref{lem:balanced},
\[
 \nupos=2n-\dim E_{\mathrm{off}+}
 =2n-2m^+(\qK|_{E_{\mathrm{off}+}})\geq2n-2p.
\]
For the second estimate combine \eqref{eq:one-sector-bound} with
orthogonality of the spectral decomposition:
\begin{align*}
 \nustar&=2n-\dim E_{\mathrm{off}+}-a_1\\
 &\geq2n-2m^+(\qK|_{E_{\mathrm{off}+}})
           -2m^+(\qK|_{E_1})-2\nu_1\\
 &\geq2n-2p-2\nu_1.
\end{align*}
The third estimate follows from \(\nustar=\nupos-a_1\).
Negative real multipliers are contained in the complementary sector, whose
dimension is at most \(2p\). If \(p=0\), that sector is zero.
\end{proof}

\section{Proofs and instability consequences}\label{sec:proofs}

All ingredients are now in place.  The variational section has recalled and
specialized the known boundary-index comparison to identify \(\widetilde k\)
with \(m^+(\qK)\), while the algebraic section converts that inertia into the
spectral counts that are the main contribution of the paper.  We now combine
these two ingredients and then spell out the resulting stability consequences.

\begin{proof}[Proof of Theorem~\ref{thm:non-autonomous}]
Proposition~\ref{prop:exact-identity} identifies
\(\widetilde k=m^+(\qK)\).  Proposition~\ref{prop:universal} therefore gives
the first inequalities in \eqref{eq:main-positive-chain} and
\eqref{eq:main-star-chain}, as well as \eqref{eq:main-algebraic}.  Expanding
the definition of \(\widetilde k\) gives the equalities in the chains, and
\(k_D,\nu_D\geq0\) gives the final inequalities.
\end{proof}

\begin{proof}[Proof of Theorem~\ref{thm:autonomous}]
As standard in autonomous variational theory (and used, for example, in
\cite{PWY22,AHPW26}), time translation preserves the Euler--Lagrange equation;
differentiation of the family \(x(t+s)\) at \(s=0\) therefore gives the Jacobi
field \(y=\dot x\). Symmetry yields \(\dot x(0)=\dot x(h)=0\), hence
\(y\in\ker(\bform|_{H_D})\). It is nonzero on \([0,h]\), because
otherwise reflection would make \(x\) constant on the whole period.
Therefore \(\nu_D\geq1\). The exact identity gives
\[
 0\leq\widetilde k=k-2k_D-\nu_D\leq k-1,
\]
so \(k\geq1\), and \eqref{eq:autonomous-positive-chain}--\eqref{eq:autonomous-star-chain} follows from
\eqref{eq:universal}. If \(a_1=2\), subtracting its contribution from the
first bound gives \eqref{eq:autonomous-a1-two}.
\end{proof}

The inequality \(k\geq1\) in this Euclidean reversible setting is consistent
with the non-minimality results of Asselle, Hu, Portaluri and Wu
\cite{AHPW26}, where natural Lagrangians on Riemannian manifolds and other
boundary conditions are treated. The present argument concerns the periodic
fixed-time index and derives the additional spectral counts from the
half-period boundary form.

\begin{cor}[Positive real expanding directions]\label{cor:growing}
Let \(\nugrow\) be the algebraic number of multipliers in \((1,\infty)\).
Then
\begin{equation}\label{eq:growing}
 \nugrow=\tfrac12\nustar\geq n-\widetilde k-\nu_1.
\end{equation}
In particular, the orbit is spectrally unstable if \(k+\nu_1<n\), or,
in the nonconstant autonomous case, if \(k+\nu_1<n+1\).
\end{cor}
\begin{proof}
The reciprocal pairing of eigenvalues of a real symplectic matrix, with
preservation of algebraic multiplicity, is standard; see
\cite[Chapter~1]{Lon02}.  Thus every positive real multiplier \(\lambda>1\)
is paired with \(\lambda^{-1}\in(0,1)\). Apply
Theorem~\ref{thm:non-autonomous} and, in the autonomous case,
\(\widetilde k\leq k-1\). A multiplier of modulus greater than one is an
exponentially growing Floquet direction.
\end{proof}

We call an orbit spectrally stable when \(\sigma(M)\subset\{z:|z|=1\}\).
Bounded linearized dynamics additionally requires semisimplicity on the unit
circle. The preceding corollary asserts spectral instability; it does not
identify the total unstable dimension, since negative or non-real unstable
multipliers can provide further growing directions.

Hyperbolicity criteria for reversible Lagrangian systems were previously
obtained by Hu, Portaluri and Yang under positivity assumptions on the
symmetry-reduced Hessian; see \cite[Theorem~4 and Corollary~2]{HPY20}.  The
next consequence has a different hypothesis and conclusion: it uses the full
fixed-period Morse index together with the algebraic multiplicity at $1$, and
it determines the number of stable and unstable transverse directions in the
index-one case.

\begin{cor}[Transverse hyperbolicity at index one]\label{cor:hyperbolicity}
Let \(n\geq2\), and let \(x\) be a nonconstant autonomous brake orbit
with \(a_1=2\). Then
\[
 \nugrow\geq n-k.
\]
Consequently spectral stability requires \(k\geq n\). If \(k=1\), all
\(2n-2\) transverse multipliers are positive real and different from
\(1\); the Poincar\'e return map on a fixed-energy transverse section is
hyperbolic, with \(n-1\) stable and \(n-1\) unstable directions.
\end{cor}
\begin{proof}
The first assertion follows from \eqref{eq:autonomous-a1-two} and reciprocal
pairing. For \(k=1\), that estimate exhausts all \(2n-2\) multipliers
outside the algebraic \(1\)-eigenspace. A nonconstant autonomous orbit lies
on a regular energy hypersurface: otherwise its Hamiltonian vector field
would vanish at some point and uniqueness would make it an equilibrium.
The standard relation between monodromy and the fixed-energy Poincar\'e map
removes the two trivial factors \((\lambda-1)\) associated with the flow and
energy directions; see, e.g., \cite[Chapter~1]{Lon02}.  Under
\(a_1=2\), its spectrum is exactly the positive real spectrum different
from \(1\). Reciprocal pairing gives the two dimensions claimed.
\end{proof}

The hypothesis \(a_1=2\) allows either a nontrivial two-dimensional Jordan
block or two semisimple trivial multipliers. Thus this corollary does not
require the full monodromy matrix to be semisimple. An index-one orbit with
additional symmetry-generated multipliers at \(1\) must instead be
assessed with the explicit nullity terms in the general estimates.

\section{Explicit examples and sharpness of the refined index}\label{sec:examples}

The purpose of this section is to make every quantity entering the main
estimates completely explicit.  Constant-coefficient oscillators are already
sufficient to exhibit the role of the half-period correction.  Although these
models are elementary, they are useful for two reasons.  First, the periodic,
Dirichlet and Neumann spectra can be written down exactly, so that the identity
\(m^+(\qK)=k-2k_D-\nu_D\) can be checked directly.  Second, products of scalar
oscillators show that the refined estimate can be sharp even when the coarser
bound involving only the full-period Morse index is entirely vacuous.

\subsection{A scalar elliptic oscillator}

Fix \(T>0\) and \(\omega>0\), and consider
\begin{equation}\label{eq:osc-L}
 L(q,v)=\frac12v^2-\frac12\omega^2q^2.
\end{equation}
The Jacobi form is
\[
 \bform_I(y,y)=\int_I\bigl(|\dot y|^2-\omega^2|y|^2\bigr)\,dt,
\]
and the Jacobi equation is \(\ddot y+\omega^2y=0\).  Put
\[
 h=\frac T2,\qquad r:=\frac{\omega T}{2\pi}.
\]
On the periodic interval \([0,T]\), the eigenvalues of
\(-d^2/dt^2-\omega^2\) are
\[
 -\omega^2\quad\text{and}\quad
 \left(\frac{2\pi m}{T}\right)^2-\omega^2,
 \qquad m\ge1,
\]
where the latter have multiplicity two.  Consequently, if \(r\notin\mathbb N\),
\begin{equation}\label{eq:osc-k-nonres}
 k=1+2\lfloor r\rfloor,
\end{equation}
whereas, if \(r=m\in\mathbb N\), the mode of frequency \(m\) is null and
\begin{equation}\label{eq:osc-k-res}
 k=2m-1,\qquad \nu_1=2.
\end{equation}
Here and below \(\nu_1\) agrees with the dimension of the periodic Jacobi
kernel.

On \([0,h]\) with Dirichlet boundary conditions, the eigenfunctions are
\(\sin(j\pi t/h)\), and the eigenvalues are
\[
 \left(\frac{j\pi}{h}\right)^2-\omega^2
 =\left(\frac{2\pi j}{T}\right)^2-\omega^2,
 \qquad j\ge1.
\]
Therefore
\begin{equation}\label{eq:osc-D}
 (k_D,\nu_D)=
 \begin{cases}
 (\lfloor r\rfloor,0),&r\notin\mathbb N,\\
 (m-1,1),&r=m\in\mathbb N.
 \end{cases}
\end{equation}
For completeness, the Neumann eigenvalues are the same numbers with
\(j\ge0\); hence
\[
 (k_N,\nu_N)=
 \begin{cases}
 (1+\lfloor r\rfloor,0),&r\notin\mathbb N,\\
 (m,1),&r=m\in\mathbb N,
 \end{cases}
\]
and indeed \(k=k_N+k_D\) and \(\nu_1=\nu_N+\nu_D\).
Combining \eqref{eq:osc-k-nonres}--\eqref{eq:osc-D} gives
\begin{equation}\label{eq:osc-tilde}
 \widetilde k=k-2k_D-\nu_D=
 \begin{cases}
 1,&r\notin\mathbb N,\\
 0,&r\in\mathbb N.
 \end{cases}
\end{equation}

Let us also verify the boundary form directly.  We use the phase coordinates
\(z=(p,y)\), with \(p=\dot y\).  Writing
\(c=\cos(\omega h)\) and \(s=\sin(\omega h)\), the half-period fundamental
matrix is
\[
 K=\begin{pmatrix}
 c&-\omega s\\[1mm]
 s/\omega&c
 \end{pmatrix}.
\]
Since \(N=\left(\begin{smallmatrix}0&1\\1&0\end{smallmatrix}\right)\), a direct
calculation yields
\begin{equation}\label{eq:osc-Q}
 \qK=N-K^TNK
 =2s\begin{pmatrix}
 -c/\omega&s\\
 s&c\omega
 \end{pmatrix}.
\end{equation}
If \(s\ne0\), then
\(\det\qK=-4s^2<0\), so \(\qK\) has inertia \((1,1,0)\) and
\(m^+(\qK)=1\).  If \(s=0\), equivalently \(r\in\mathbb N\), then
\(\qK=0\) and \(m^+(\qK)=0\).  Thus \eqref{eq:osc-tilde} is exactly the
boundary-inertia identity.

Finally,
\[
 M=\begin{pmatrix}
 \cos(\omega T)&-\omega\sin(\omega T)\\
 \sin(\omega T)/\omega&\cos(\omega T)
 \end{pmatrix},
\qquad \sigma(M)=\{e^{2\pi i r},e^{-2\pi i r}\}.
\]
If \(r\in\mathbb N\), then \(M=I_2\), \(\nupos=2\), \(\nustar=0\),
\(\nu_1=a_1=2\), and the first estimate is the equality
\(2=2-2\widetilde k\).  If \(r\notin\mathbb N\) and
\(r\notin\frac12\mathbb Z\), both multipliers are nonreal and
\(\nupos=\nustar=0=2-2\widetilde k\).  At a nonintegral half-integer the two
multipliers equal \(-1\), so the same equality for \(\nupos\) remains true.
Thus the refined positive-real bound is sharp for every scalar oscillator.

\subsection{Decoupled oscillators: a family for which the coarse bound is vacuous}

Consider now on \(\mathbb R^n\)
\begin{equation}\label{eq:multi-osc}
 L(q,v)=\frac12|v|^2-\frac12\sum_{j=1}^n\omega_j^2q_j^2,
 \qquad \omega_j>0,
\end{equation}
and put \(r_j=\omega_jT/(2\pi)\).  All indices and all spectral counts split
as sums of their scalar contributions.  Let
\[
 R:=\#\{j:r_j\in\mathbb N\},\qquad
 Q:=n-R=\#\{j:r_j\notin\mathbb N\}.
\]
From the preceding computation,
\begin{align}\label{eq:multi-indices}
 k&=\sum_{r_j\notin\mathbb N}\bigl(1+2\lfloor r_j\rfloor\bigr)
   +\sum_{r_j=m_j\in\mathbb N}(2m_j-1),\\
 k_D&=\sum_{r_j\notin\mathbb N}\lfloor r_j\rfloor
   +\sum_{r_j=m_j\in\mathbb N}(m_j-1),\nonumber\\
 \nu_D&=R,\qquad
 \widetilde k=k-2k_D-\nu_D=Q.\nonumber
\end{align}
Moreover \(\qK\) is the orthogonal direct sum of the scalar forms
\eqref{eq:osc-Q}; hence
\[
 m^+(\qK)=Q=\widetilde k.
\]
For each scalar block, positive real multipliers occur exactly when
\(r_j\in\mathbb N\), in which case the block is \(I_2\).  A nonintegral
half-integer contributes the pair \((-1,-1)\), while every other nonintegral
value contributes a nonreal conjugate pair.  Therefore
\begin{equation}\label{eq:multi-counts}
 \nupos=2R=2n-2Q=2n-2\widetilde k.
\end{equation}
Thus the refined bound is an equality throughout the whole family.
Furthermore
\[
 \nu_1=a_1=2R,\qquad \nustar=0,
\]
and the estimate involving the algebraic multiplicity at one is also exact:
\[
 \nustar=0=2n-2\widetilde k-a_1.
\]

The gain over the coarse estimate can be arbitrarily large.  Indeed, increasing
any \(\omega_j\) through successive oscillation bands increases \(k\) and
\(k_D\) simultaneously while leaving \(\widetilde k\) equal to either zero or
one on that block.  Hence \(2n-2k\) may tend to \(-\infty\), while the refined
right-hand side \(2n-2\widetilde k\) continues to give the exact nonnegative
spectral count.

A concrete three-dimensional instance makes this especially transparent.
Take
\begin{equation}\label{eq:concrete-ratios}
 (r_1,r_2,r_3)=\left(1,\frac32,2\right).
\end{equation}
The three scalar contributions to \((k,k_D,\nu_D,\widetilde k)\) are
\[
 (1,0,1,0),\qquad (3,1,0,1),\qquad (3,1,1,0),
\]
so that
\begin{equation}\label{eq:concrete-indices}
 n=3,\qquad k=7,
 \qquad k_D=2,
 \qquad \nu_D=2,
 \qquad \widetilde k=1.
\end{equation}
The monodromy is the direct sum of \(I_2\), \(-I_2\), and \(I_2\).  Hence
\begin{equation}\label{eq:concrete-counts}
 \nupos=4,\qquad \nustar=0,\qquad \nu_1=a_1=4.
\end{equation}
The estimate using only the full Morse index reads
\(\nupos\ge 2n-2k=-8\) and is completely vacuous.  In contrast, the refined
estimate gives
\[
 \nupos\ge2n-2\widetilde k=6-2=4,
\]
which is the exact value.  Likewise,
\[
 \nustar\ge2n-2\widetilde k-a_1=6-2-4=0
\]
is sharp.  This example isolates the precise information retained by
\(k-2k_D-\nu_D\) and lost by the full-period index \(k\) alone.

\subsection{A nonconstant autonomous brake orbit}

The preceding three-dimensional example can be realized along an actual
nonconstant brake orbit of the autonomous Lagrangian \eqref{eq:multi-osc}.
Choose \(\omega_1>0\), set
\[
 \omega_2=\frac32\omega_1,\qquad \omega_3=2\omega_1,
 \qquad T=\frac{2\pi}{\omega_1},
\]
and consider
\begin{equation}\label{eq:brake-example}
 x(t)=\bigl(A\cos(\omega_1t),0,0\bigr),\qquad A\ne0.
\end{equation}
Then \(x(T-t)=x(t)\), while \(\dot x(0)=\dot x(T/2)=0\); hence the two
endpoints of the half orbit are brake instants.  Since the Lagrangian is
quadratic, its Jacobi equation along \(x\) is exactly the decoupled system
already computed.  Therefore all the numbers in
\eqref{eq:concrete-indices}--\eqref{eq:concrete-counts} apply unchanged:
\[
 k=7,\quad k_D=2,\quad \nu_D=2,\quad \widetilde k=1,
 \quad \nupos=4,\quad \nustar=0,\quad \nu_1=a_1=4.
\]
The Dirichlet kernel has dimension two.  One of its generators is the
mandatory time-translation field
\[
 \dot x(t)=\bigl(-A\omega_1\sin(\omega_1t),0,0\bigr),
\]
which vanishes at \(0\) and \(T/2\).  The second comes from the resonant third
normal mode \(y_3(t)=\sin(2\omega_1t)\).  Thus this example also shows
concretely that the autonomous inequality \(\nu_D\ge1\) may be strict: extra
Dirichlet null directions can arise from transverse resonances.

It is worth stressing that the large Morse index \(k=7\) does not contradict
the presence of four positive real multipliers.  The two resonant elliptic
blocks contribute four multipliers at \(+1\), whereas the middle block
contributes two multipliers at \(-1\).  The full index counts all negative
variational directions accumulated by the three oscillators; the corrected
quantity \(\widetilde k=1\), instead, detects exactly the single block whose
multipliers fail to be positive real.  This is the mechanism behind the
refined theorem in its simplest fully computable form.

\subsection{A hyperbolic scalar block}

For comparison, take
\[
 L(q,v)=\frac12v^2+\frac12\alpha^2q^2,\qquad \alpha>0.
\]
Then
\[
 \bform_I(y,y)=\int_I\bigl(|\dot y|^2+\alpha^2|y|^2\bigr)\,dt
\]
is positive definite for periodic, Neumann and Dirichlet boundary conditions.
Thus
\[
 k=k_N=k_D=0,\qquad \nu_N=\nu_D=\nu_1=0,
 \qquad \widetilde k=0.
\]
With \(c_h=\cosh(\alpha h)\), \(s_h=\sinh(\alpha h)\),
\[
 K=\begin{pmatrix}c_h&\alpha s_h\\s_h/\alpha&c_h\end{pmatrix},
\qquad
 \qK=-2s_h\begin{pmatrix}c_h/\alpha&s_h\\s_h&\alpha c_h\end{pmatrix}.
\]
The matrix in the last display is positive definite because its determinant is
\(c_h^2-s_h^2=1\); hence \(\qK\) is negative definite and
\(m^+(\qK)=0=\widetilde k\).  The full monodromy has eigenvalues
\(e^{\alpha T}\) and \(e^{-\alpha T}\), so
\[
 \nupos=\nustar=2,\qquad a_1=0.
\]
Both main spectral bounds are equalities.  This elementary block is a useful
reminder that ``positive real'' is a location statement, not a stability
statement: the pair is hyperbolic despite the vanishing Morse index.

\subsection{A nonlinear reversible example: the H\'enon--Heiles $A$-orbit}

The constant-coefficient examples above make the index formulas completely
transparent, but it is also useful to record a genuinely nonlinear model in
which the same half-period mechanism is visible without any reduction to
normal modes. A classical example is provided by the H\'enon--Heiles
Hamiltonian, introduced in \cite{HH64}; the stability and bifurcation theory
of its straight-line orbit has a large literature, and the Lam\'e-equation
analysis relevant to the transverse variational equation may be found in
\cite{BMT01}.
\begin{equation}\label{eq:HH-H}
 H(x,y,p_x,p_y)=\frac12(p_x^2+p_y^2)+V(x,y),
 \qquad
 V(x,y)=\frac12(x^2+y^2)+x^2y-\frac13y^3.
\end{equation}
Equivalently, one may work with the reversible natural Lagrangian
\begin{equation}\label{eq:HH-L}
 L(q,v)=\frac12|v|^2-V(q),\qquad q=(x,y)\in\R^2,
\end{equation}
which satisfies $L(q,-v)=L(q,v)$.

The vertical line $\{x=0\}$ is invariant because $\partial_xV(0,y)=0$. Along
that line the dynamics reduces to the scalar equation
\begin{equation}\label{eq:HH-y}
 \ddot y+y-y^2=0,
 \qquad
 E=\frac12\dot y^2+W(y),
 \qquad
 W(y):=\frac12y^2-\frac13y^3.
\end{equation}
For every energy $0<E<1/6$, the equation $W(y)=E$ has three real roots.
The bounded component of the energy level lies between the negative root
$y_-(E)<0$ and the smaller positive root $y_+(E)\in(0,1)$ and gives a
nonconstant periodic libration.  Denoting this solution by $y_E(t)$ and its
period by $T_E$, the configuration-space orbit
\[
 x_E(t)=(0,y_E(t))
\]
is a periodic brake orbit of \eqref{eq:HH-L}: after a time shift we may assume
\[
 \dot y_E(0)=\dot y_E(T_E/2)=0,
 \qquad
 x_E(T_E-t)=x_E(t).
\]
The half-period interval $[0,T_E/2]$ is therefore the natural domain for the
Dirichlet/Neumann decomposition used throughout the paper.

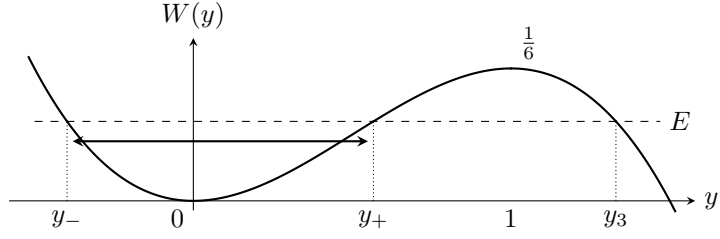
\begin{figure}[t]
\centering
\begin{tikzpicture}[x=4.2cm,y=10.5cm,>=stealth]
  \draw[->] (-0.58,0) -- (1.58,0) node[right] {$y$};
  \draw[->] (0,-0.012) -- (0,0.205) node[above] {$W(y)$};
  \draw[thick,smooth,domain=-0.52:1.52,samples=220]
       plot (\x,{0.5*\x*\x-\x*\x*\x/3});
  % Illustrative energy E=0.1; the three roots are approximately
  % -0.3976, 0.5671, 1.3305.
  \draw[dashed] (-0.50,0.10) -- (1.47,0.10) node[right] {$E$};
  \draw[densely dotted] (-0.3976,0) -- (-0.3976,0.10);
  \draw[densely dotted] (0.5671,0) -- (0.5671,0.10);
  \draw[densely dotted] (1.3305,0) -- (1.3305,0.10);
  \fill (-0.3976,0.10) circle (0.45pt);
  \fill (0.5671,0.10) circle (0.45pt);
  \fill (1.3305,0.10) circle (0.45pt);
  \node[below] at (-0.3976,0) {$y_-$};
  \node[below] at (0.5671,0) {$y_+$};
  \node[below] at (1.3305,0) {$y_3$};
  \node[below left] at (0,0) {$0$};
  \node[below] at (1,0) {$1$};
  \fill (1,0.1666667) circle (0.45pt);
  \node[above right] at (1,0.1666667) {$\frac16$};
  \draw[<->,thick] (-0.38,0.075) -- (0.55,0.075);
\end{tikzpicture}
\caption{The effective potential $W(y)=\tfrac12y^2-\tfrac13y^3$ on the invariant line $x=0$.  For $0<E<\tfrac16$, the bounded $A$-orbit oscillates between the negative turning point $y_-$ and the smaller positive turning point $y_+$.  The third root $y_3>1$ belongs to the unbounded branch and is not a turning point of the periodic libration.}
\label{fig:HH-potential}
\end{figure}

Along the $A$-orbit, the Hessian of the potential is diagonal:
\begin{equation}\label{eq:HH-hessian}
 D^2V(0,y_E(t))=
 \begin{pmatrix}
 1+2y_E(t)&0\\[1mm]
 0&1-2y_E(t)
 \end{pmatrix}.
\end{equation}
Hence the Jacobi equation splits into a transverse and a longitudinal scalar
problem. Writing a variation as $u=(\xi,\eta)$, the second variation on any
interval $I$ becomes
\begin{equation}\label{eq:HH-splitting-form}
 \bform_I(u,u)=
 \int_I\Bigl(\dot\xi^2-(1+2y_E(t))\xi^2\Bigr)\,\dd
 +\int_I\Bigl(\dot\eta^2-(1-2y_E(t))\eta^2\Bigr)\,\dd.
\end{equation}
Accordingly, on the half-period we introduce
\begin{equation}\label{eq:HH-ops}
 \cA_{\perp,E}:=-\frac{d^2}{dt^2}-(1+2y_E(t)),
 \qquad
 \cA_{\parallel,E}:=-\frac{d^2}{dt^2}-(1-2y_E(t)).
\end{equation}
The first operator governs transverse variations of the invariant line
$x=0$, whereas the second governs variations tangent to the orbit.

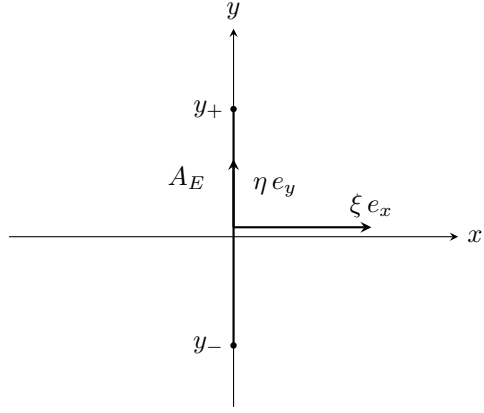
\begin{figure}[t]
\centering
\begin{tikzpicture}[x=1.35cm,y=1.25cm,>=stealth]
  \draw[->] (-2.2,0) -- (2.2,0) node[right] {$x$};
  \draw[->] (0,-1.8) -- (0,2.2) node[above] {$y$};
  \draw[thick] (0,-1.15) -- (0,1.35);
  \fill (0,-1.15) circle (1.2pt);
  \fill (0,1.35) circle (1.2pt);
  \node[left] at (0,-1.15) {$y_-$};
  \node[left] at (0,1.35) {$y_+$};
  \draw[->,thick] (0,0.10) -- (0,0.82);
  \node[right] at (0.10,0.52) {$\eta\,e_y$};
  \draw[->,thick] (0,0.10) -- (1.35,0.10) node[above] {$\xi\,e_x$};
  \node[left] at (-0.18,0.62) {$A_E$};
\end{tikzpicture}
\caption{Configuration-space picture of the H\'enon--Heiles $A$-orbit.  During one half-period the orbit moves on the invariant line $x=0$ from one brake instant to the other.  Along the orbit, variations split into the longitudinal component $\eta e_y$ and the transverse component $\xi e_x$.}
\label{fig:HH-geometry}
\end{figure}

The longitudinal equation possesses a canonical Dirichlet Jacobi field.
Indeed, differentiating \eqref{eq:HH-y} gives
\[
 \dddot y_E +(1-2y_E)\dot y_E=0,
\]
so that $\eta=\dot y_E$ solves
\[
 \ddot\eta +(1-2y_E(t))\eta=0.
\]
Since $\dot y_E(0)=\dot y_E(T_E/2)=0$, this field belongs to the
half-period Dirichlet kernel.

\begin{prop}\label{prop:HH-Dirichlet}
For every nonconstant H\'enon--Heiles $A$-orbit at energy $0<E<1/6$, the
half-period Dirichlet nullity satisfies
\[
 \nu_D\ge 1.
\]
More precisely, the longitudinal Dirichlet problem for $\cA_{\parallel,E}$ has
one-dimensional kernel spanned by $\dot y_E$, and the full half-period
Dirichlet problem splits as
\[
 \ker \cA_D
 =\ker (\cA_{\perp,E})_D \oplus \mathrm{span}\{(0,\dot y_E)\}.
\]
Consequently,
\[
 \nu_D=1+\dim\ker (\cA_{\perp,E})_D,
 \qquad
 k_D=k_{D,\perp}+k_{D,\parallel},
 \qquad
 k=k_{\perp}+k_{\parallel}.
\]
\end{prop}
\begin{proof}
The orthogonal decomposition of the quadratic form is exactly
\eqref{eq:HH-splitting-form}, hence both the Dirichlet and the periodic
operators decompose into transverse and longitudinal scalar Sturm--Liouville
problems. The field $(0,\dot y_E)$ belongs to the Dirichlet kernel by the
preceding computation. Since the longitudinal problem is scalar, any
Dirichlet kernel has dimension at most one; therefore
$\ker(\cA_{\parallel,E})_D=\mathrm{span}\{\dot y_E\}$. The formulas for
$\nu_D$, $k_D$ and $k$ are then immediate from the orthogonal splitting.
\end{proof}

Proposition~\ref{prop:HH-Dirichlet} gives a concrete geometric realization of
the autonomous correction term in Theorem~\ref{thm:autonomous}. The
extra Dirichlet null direction is no longer an abstract time-translation
symmetry: in Figure~\ref{fig:HH-geometry} it is simply the tangential
variation of the straight-line brake orbit. Any additional contribution to
$\nu_D$ must come from the transverse operator $(\cA_{\perp,E})_D$.  The same
transverse variational equation is also the classical starting point for the
stability and bifurcation analysis of the $A$-orbit; see, for example, the
Lam\'e-equation treatment in \cite{BMT01}.  Thus the H\'enon--Heiles model
exhibits, in a genuinely nonlinear setting, the longitudinal/transverse
splitting underlying the half-period correction used in our theorem.

We deliberately do not tabulate a numerical value of $(k,k_D,\nu_D)$ for a
specific energy here. Such a table would require an independent numerical or
computer-assisted determination of the transverse Dirichlet spectrum and of
the corresponding Floquet multipliers. The structural point of the example is
instead that the H\'enon--Heiles $A$-orbit fits the hypotheses of the
autonomous theorem exactly, and that the splitting responsible for the refined
index formula can be seen explicitly at the level of the Jacobi equation.

\section{Comparison with the Lin--Zeng instability index}
\label{sec:lin-zeng}

The instability-index theorem of Lin and Zeng \cite{LZ22} is close in spirit
to the present work in that it relates the inertia of a symmetric form to
Hamiltonian spectral information.  The two results, however, address
different linear problems and use different indices.  We record the
distinction because identifying them formally would be misleading.

\subsection{Time-independent generator versus periodic monodromy}

Lin and Zeng consider an autonomous linear Hamiltonian equation
\begin{equation}\label{eq:LZ-generator}
   \frac{du}{dt}=J_0L_0u
\end{equation}
on a real Hilbert space.  Here $J_0$ and $L_0$ are fixed in time; $J_0$ may be
unbounded, while $L_0:X\to X^*$ induces a bounded symmetric bilinear form with
finite negative index and a coercive positive complement.  The spectral
quantities in their theorem are therefore attached to the single generator
$J_0L_0$.  Their theory is designed in particular for Hamiltonian PDEs and
allows continuous spectrum and embedded imaginary eigenvalues.

In the present paper the linear Hamiltonian equation is obtained by
linearizing along a generally nonconstant $T$-periodic solution.  In the
notation of \eqref{eq:Hamiltonian-system} it is
\begin{equation}\label{eq:periodic-LZ-comparison}
   \dot z=JB(t)z,\qquad B(t+T)=B(t),
\end{equation}
with $B(t)$ depending on the reference orbit.  Even when the original
Lagrangian is autonomous, $B(t)$ is generally time dependent along a
nonconstant periodic orbit.  The natural spectral object is consequently the
monodromy matrix
\[
   M=\Phi(T),
\]
not the spectrum of one time-independent generator.  In general there is no
fixed matrix $A$ for which $M=e^{TA}$ in a way compatible with the original
variational problem, and one certainly cannot replace the time-ordered
evolution by
\[
   \exp\!\left(\int_0^T JB(t)\,dt\right)
\]
without additional commutation assumptions.

Thus the two spectral problems are already different at the linear level:
Lin--Zeng count spectral data of an autonomous generator, while our estimates
count Floquet multipliers of a periodic nonautonomous system.

\subsection{Energy index versus action Morse index}

Under hypotheses (H1)--(H3), Lin and Zeng prove the identity
\begin{equation}\label{eq:lin-zeng}
 k_r+2k_c+2k_i^{\leq0}+k_0^{\leq0}=n^-(L_0),
\end{equation}
where $k_r$ and $k_c$ count generalized eigenspaces of $J_0L_0$ in the
positive real and first-quadrant sectors, while the remaining terms measure
the nonpositive inertia of $L_0$ on the relevant imaginary and zero spectral
subspaces; see \cite[Theorem~2.3 and Section~2.4]{LZ22}.  The quantity on the
right is the negative index of the time-independent energy form $L_0$.

Our indices have a different origin.  The number
\[
   k=m^-(\bform_{[0,T]}|_{H_P})
\]
is the Morse index of the fixed-period action Hessian, and $k_D,\nu_D$ come
from its half-period Dirichlet realization.  The refined quantity entering our
Floquet estimate is
\[
   \widetilde k=k-2k_D-\nu_D=m^+(\qK).
\]
There is no general reason for either $k$ or $\widetilde k$ to coincide with
$n^-(L_0)$.  The difference is not merely notational: one index is attached to
a fixed Hamiltonian energy operator, while the others are variational indices
of a periodic orbit and of its boundary-value problems.

A simple equilibrium illustrates the distinction.  Let
\[
   L(q,v)=\frac12v^2+\frac12a^2q^2,\qquad a>0,
\]
and $x\equiv0$.  The periodic action Hessian
\[
   \int_0^T(\dot y^2+a^2y^2)\,dt
\]
is positive definite, so $k=0$.  The corresponding Hamiltonian is
$H(p,q)=\frac12p^2-\frac12a^2q^2$; its energy Hessian has one negative
direction, and the Hamiltonian generator has eigenvalues $\pm a$.  Thus even
in the autonomous constant-coefficient case the action Morse index and the
energy index need not agree.

Accordingly, we do not use the Lin--Zeng theorem in the proof of our results,
and we do not claim a canonical identification between the two index theories.
They are complementary: Lin--Zeng gives an instability index for a time-independent
Hamiltonian generator, including infinite-dimensional problems, whereas the
present estimates use reversibility and half-period action data to control the
Floquet spectrum of a periodic linearization.  A systematic relation between
$n^-(L_0)$ and the indices $k,k_D,\nu_D$ would require additional structure
and lies outside the scope of this paper.

\section{Concluding perspective}

The main outcome of the paper is a quantitative constraint on the Floquet
spectrum of a reversible periodic orbit.  Once the known half-period
Morse-index comparison is specialized to the Neumann--Dirichlet splitting, its
boundary inertia $\widetilde k=k-2k_D-\nu_D$ controls how many multipliers can
leave the positive real axis.  The new algebraic estimate is formulated with
algebraic multiplicities and retains arbitrary Jordan structure at the unit
multiplier.

Two immediate consequences clarify the stability content.  For $k=0$ the
estimate recovers the positive-real-spectrum theorem for reversible minimizers
\cite{Ure18}.  For a nonconstant autonomous brake orbit, time translation
forces $\nu_D\ge1$ and improves the spectral bounds by two; when $a_1=2$, an
index-one orbit is transversely hyperbolic, with $n-1$ stable and $n-1$
unstable directions on a fixed-energy Poincar\'e section.

The examples show that the half-period correction is not cosmetic.  In the
resonant oscillator family the bound based only on the full Morse index can be
vacuous while the refined estimate is exact.  The H\'enon--Heiles $A$-orbit
shows, in a genuinely nonlinear system, how the same half-period data arise
from the longitudinal/transverse splitting of the Jacobi equation.  The
comparison with Lin--Zeng also delineates the scope of the result: our theorem
is a fixed-time Floquet estimate for a generally time-periodic linearization,
not an energy-index theorem for a time-independent Hamiltonian generator.

\section*{Declarations}

\noindent\textbf{Funding.} A.P. gratefully acknowledges support from
INdAM--GNAMPA and FERA\_RILO23.
L.W. is partially supported by the National Natural Science Foundation
of China (NSFC No.\ 12171281).

\smallskip
\noindent\textbf{Declaration of interests.} The authors declare that they have no known competing financial interests or personal relationships that could have appeared to influence the work reported in this paper.

\smallskip
\noindent\textbf{Data availability.} No data sets were generated or analysed during the present study.

\smallskip
\noindent\textbf{Use of generative artificial intelligence.} During the preparation of this manuscript, the authors used ChatGPT (OpenAI) for language editing, organizational assistance, literature-search support, and consistency checks.  All mathematical statements, derivations, proofs, citations, and references were independently checked and approved by the authors, who take full responsibility for the content of the article.

\bigskip
\noindent Prof. Alessandro Portaluri\\
Universit\`a degli Studi di Torino (DISAFA)\\
Largo Paolo Braccini 2, 10095 Grugliasco, Torino, Italy\\
\href{mailto:alessandro.portaluri@unito.it}{\texttt{alessandro.portaluri@unito.it}}

\medskip
\noindent Prof. Li Wu\\
School of Mathematics, Shandong University\\
Jinan, 250100, P.~R.~China\\
\href{mailto:vvvli@sdu.edu.cn}{\texttt{vvvli@sdu.edu.cn}}

\end{document}